\documentclass[10pt]{article}

\usepackage{mathtext}
\usepackage[utf8]{inputenc}
\usepackage[english]{babel}

\usepackage{amsfonts, amsmath, amssymb, amsthm, amscd}
\usepackage{mathtools}
\usepackage{xcolor}
\usepackage{ulem}

\usepackage{enumerate}

\usepackage{geometry}
\usepackage{graphicx}
\usepackage{hyperref}
\usepackage{cancel}
\hypersetup{pdfstartview=FitH,  linkcolor=blue,urlcolor=urlcolor, citecolor=blue, colorlinks=true}

\DeclarePairedDelimiter\abs{\lvert}{\rvert}
\newcommand{\norm}[1]{\left\lVert#1\right\rVert}

\DeclareMathOperator{\supp}{supp}
\DeclareMathOperator{\real}{Re}
\DeclareMathOperator{\im}{Im}
\DeclareMathOperator{\Ran}{Ran}

\theoremstyle{theorem}
\newtheorem{theorem}{Theorem}[section]
\theoremstyle{theorem}
\newtheorem{remark}{Remark}[section]
\theoremstyle{theorem}
\newtheorem{proposition}{Proposition}[section]
\theoremstyle{lemma}
\newtheorem{lemma}{Lemma}[section]
\theoremstyle{definition}
\newtheorem{definition}{Definition}[section]

\theoremstyle{corollary}
\newtheorem{corollary}{Corollary}[section]
\newcommand{\Label}[1]{\label{#1}}

\usepackage{fancyhdr}
\newcommand{\CCx}{\mathbb{C^\times}}
\newcommand{\rank}{\mathbb{\text{{\rm rank}\,}}}
\newcommand{\MMpx}{\mathfrak{M}^{-1}_{\pi\to x}}
\newcommand{\MMxp}{\mathfrak{M}_{x\to\pi}}

\newcommand{\be}{\beta}

\numberwithin{equation}{section}

\title{A Pseudo-Differential Operator of the $p$-Adic Multiplicative Calculus}

\author{\textbf{Alexandra V. Antoniouk}\\
\footnotesize Institute of Mathematics,\\
\footnotesize National Academy of Sciences of Ukraine,\\
\footnotesize Tereshchenkivska 3, Kyiv, 01024 Ukraine,\\
\footnotesize American University Kyiv, Poshtova Sq 3, 04070, Kyiv, Ukraine
\footnotesize E-mail: antoniouk.a@gmail.com
\and
\textbf{Anatoly N. Kochubei}\\
\footnotesize Institute of Mathematics,\\
\footnotesize National Academy of Sciences of Ukraine,\\
\footnotesize Tereshchenkivska 3, Kyiv, 01024 Ukraine,\\
\footnotesize E-mail: kochubei@imath.kiev.ua
\and
\textbf{Mariia V. Serdiuk}\\
\footnotesize Taras Shevchenko National University of Kyiv,\\
\footnotesize  Akademika Hlushkova Ave 4-e, Kyiv, 03127 Ukraine,\\
\footnotesize E-mail: mariia.v.serdiuk@gmail.com }

\begin{document}

\maketitle

\large
\begin{abstract}
We introduce and study a pseudo-differential operator $W^\alpha$, defined via the p-adic Mellin transformation, acting on real- or complex-valued functions on the multiplicative group of the field of $p$-adic numbers. The operator $W^\alpha$ is maximally dissipative. An estimate of its heat kernel is given.
\end{abstract}

\section{Introduction}

By the classical Ostrowski theorem, the field $\Bbb Q_p$ of $p$-adic numbers (here $p$ is a prime number) is the only alternative to $\mathbb R$ in terms of possible absolute values on the field $\mathbb Q$ of rational numbers. This fact,
together with various applications of $\mathbb Q_p$, makes the study of $p$-adic functions an important branch of mathematical analysis.

The $p$-adic analysis is divided essentially into the study of functions $\mathbb Q_p\to \mathbb Q_p$ (polynomials, analytic functions etc) and that of functions  on $\mathbb Q_p$ and various $p$-adic spaces, with values from $\mathbb R$
and $\mathbb C$ (group characters, integral transforms, wave functions in $p$-adic models of mathematical physics etc). Such a function cannot be differentiated. However, for such functions the Fourier analysis is available, making it possible
to define pseudo-differential operators  ($\Psi$DOs).

Below we consider various objects of this, second kind of $p$-adic analysis. The model object here is the $\Psi$DO $D^\alpha$ called the Vladimirov-Taibleson operator. It can be defined the Fourier analysis as a $\Psi$DO with the symbol $\xi \to |\xi|^\alpha$. An alternative is an explicit representation as a hypersingular integral operator of fractional differentiation. We would like to stress that the investigation of $D^\alpha$, now well-developed, has lead to the study, often highly nontrivial, of various classes of functions, now seen as important branches of $p$-adic analysis.

In particular, we have to mention various analogs of the Sobolev spaces \cite{T1975,Z,Z2025}, orthonormal systems constructed by Vladimirov et al \cite{VVZ,K2001}, $p$-adic wavelets \cite{AKS,KKZ}, $p$-adic radial calculus and integral equations \cite{K2014,AKS1}, constructions of $p$-adic geometry \cite{Be}. Note also the theory of $p$-adic parabolic equations and Markov processes \cite{K2001,Z,Z2025}, $p$-adic wave equation \cite{K2008}, various types of nonlinear equations involving $D^\alpha$ \cite{KK,AKN}.

In this paper, we initiate a study of a $\Psi$DO $W^\alpha$ on $\mathbb Q_p$ defined on functions with a singularity and simultaneously closely related to the algebraic structures of $\mathbb Q_p$. This is a $\Psi$DO on $\mathbb Q_p\setminus \{ 0\}$, the multiplicative group of $\mathbb Q_p$. The definition is based on the $p$-adic Mellin transform (see \cite{GG,T1975}). The operator $W^\alpha$ is maximally dissipative. We consider the corresponding analog of the heat equation and give an estimate of the heat kernel. This estimate looks quite different from its counterparts, both in real and $p$-adic analysis.

\section{Preliminaries}

Let $\mathbb{Q}_p$ be the field of $p$--adic numbers, $\mathbb{Q}_p^{\times}=\mathbb{Q}_p\setminus\{0\}$ its multiplicative group. On $\mathbb{Q}_p^{\times}$ the Haar measure $dx^{\times}$ is defined as $\dfrac{dx}{|x|_p}$, where $dx$ is the Haar measure on the additive group $\mathbb{Q}_p$.

We denote by $B_n(a)=\{x\in\mathbb{Q}_p:|x-a|_p\leq p^n\}$ a ball with the center at the point $a$ of radius $p^n$, $S_n(a)=\{x\in\mathbb{Q}_p:|x-a|_p=p^n\}$ a sphere with the center at $0$ of radius $p^n$ (respectively, $B_n$ and $S_n$ in case $a=0$). The sphere $S_0 =\{x\in\mathbb{Q}_p:|x|_p=1\}$ is a multiplicative group of units. 

Consider the set $A=\{x:|1-x|_p<1\}$. Note that the set $A$ is a compact subgroup of $\mathbb{Q}_p^{\times}$, and every element $x\in\mathbb{Q}_p^{\times}$ can be uniquely decomposed into a product $x=p^n\varepsilon^{\ell}a$, where $|x|_p=p^{-n}$, $0\leq\ell\leq p-2$ and $a\in A$. Therefore, $\mathbb{Q}_p^{\times}$ is isomophic to the direct product of three groups:
\begin{equation}\Label{Qpx}\
\mathbb{Q}_p^{\times}\cong\mathbb{Z}\times\mathbb{Z}_{p-1}\times A,
\end{equation}
where $\mathbb{Z}$ denotes the infinite cyclic group of the elements $\{p^n\}_{n=-\infty}^{n=\infty}$, $\mathbb{Z}_{p-1}$ denotes the finite cyclic group of the elements $\{\varepsilon^{\ell},0\leq\ell\leq p-2\}$.
Or equivalently, $\mathbb{Q}_p^{\times}$ can also be viewed as
\begin{equation}\Label{Qpx1}\
\mathbb{Q}_p^{\times}\cong p^\mathbb{Z}\times S_0.
\end{equation}

Let $A_0=S_0$, $A_n=1+B_{-n}=B_{-n}(1)=1+p^n\mathbb{Z}_p$, $n\geq1$. The family $\{A_n\}_{n=0}^{\infty}$ is a fundamental system of neighbourhoods of the unit in $\mathbb{Q}_p^{\times}$.

\begin{definition}
We call \textit{a multiplicative character} over the field $\mathbb{Q}_p$ a character of the multiplicative group $\mathbb{Q}_p^{\times}$, i.e. a continuous group homomorphism $\pi:\mathbb{Q}_p^{\times}\to\CCx$:
\begin{equation}\Label{pi}\
\pi(xy)=\pi(x)\pi(y)\quad\mbox{ for all }x,y\in\mathbb{Q}_p^{\times}.
\end{equation}
Respectively, we call \textit{a unitary multiplicative character} over the field $\mathbb{Q}_p$ such multiplicative character $\pi$ that
$$|\pi(x)|=1\quad\mbox{ for all }\quad x\in\mathbb{Q}_p^{\times}.$$
\end{definition}

The set of all unitary multiplicative characters on $\mathbb{Q}_p^{\times}$ forms a dual group $\widehat{\mathbb{Q}_p^{\times}}$ with the operation of pointwise multiplication $(\pi_1\pi_2)(x)=\pi_1(x)\pi_2(x)$, $x\in\mathbb{Q}_p^{\times}$.

\begin{definition}
A multiplicative character $\pi\in\widehat{\mathbb{Q}_p^{\times}}$ is called \textit{unramified} if it is trivial on $S_0$, i.e.
$$\pi(u)=1\mbox{ for all }u\in S_0.$$

If the character is not trivial on $S_0$, it is called \textit{ramified}.
\end{definition}

For each character $\pi\in\widehat{\mathbb{Q}_p^{\times}}$ there is an integer $k\geq0$ such that $\pi$ is trivial on $A_k$ (see \cite{T1975}). If $\pi$ is unmarified, we say that it has \textit{a rank} $0$. If for some $k\geq1$ $\pi$ is trivial on $A_k$, but not on $A_{k-1}$, we say that $\pi$ has \textit{rank} $k$. We will denote the character's rank as $\rank\pi$.

It follows from the decomposition \eqref{Qpx} that the dual group $\widehat{\mathbb{Q}_p^{\times}}$ is a direct product of three groups: the group of characters that depend only on the norm of the argument, that is $\pi(x)=|x|_p^{i\beta}$, $x\in\mathbb{Q}_p^{\times}$, $-\frac{\pi}{\ln p}<\beta\leq \frac{\pi}{\ln p}$, the cyclic group of the order $p-1$ (characters which depend on $\varepsilon^{\ell}$), and the infinite discrete group $\widehat{A}$, which is dual to $A$. There exists only a finite number of characters on $A$ of each rank, so the group $\widehat{A}$ is countable. Any character $\pi\in\widehat{\mathbb{Q}_p^{\times}}$ can be decomposed as $\pi=\theta|\cdot|_p^{i\beta}$, where $\beta$ is a real parameter, $\theta$ is a multiplicative character on the unit spehere $S_0=\mathbb{Z}_{p-1}\times A$.

We will denote the dual group of all characters $\theta$ on the multiplicative group $S_0$ as $\widehat{S_0}$.

We will use the following integration formulas (see \cite{Vlad}). Let $dx$ be the additive Haar measure on $\mathbb{Q}_p$, $\gamma\in\mathbb{Z}$. Then
\begin{align}\Label{Vlad-11.3}
\int\limits_{S_\gamma}dx&=(1-p^{-1})p^{\gamma};\\
\Label{Vlad-11.19}
\int\limits_{S_0}|x-1|_p^{\alpha-1}dx&=\frac{p-2+p^{-\alpha}}{p(1-p^{-\alpha})},\;\real\alpha>0.
\end{align}
If $\pi\in\widehat{\mathbb{Q}_p^{\times}}$, $\rank\pi=k\geq1$, then
\begin{align}\Label{Vlad-11.40}
\int\limits_{S_{\gamma}}\pi(x)dx&=0;\\
\Label{Vlad-11.44}
\int\limits_{S_0}|x-1|_p^{\alpha-1}\pi(x)dx&=\Gamma_p(\alpha)p^{-\alpha k},\;\real\alpha>0,
\end{align}
where $\Gamma_p(\alpha)=\dfrac{1-p^{\alpha-1}}{1-p^{-\alpha}}$, $\alpha\in\mathbb{C}$, $\alpha\neq\dfrac{2k\pi i}{\ln p}$, $k\in\mathbb{Z}$, is the $p$--adic gamma function.

\begin{definition}
\textit{The set of test functions} $S^{\times}$ on $\mathbb{Q}_p^{\times}$ is the set of functions $\varphi:\mathbb{Q}_p^{\times}\to\mathbb{C}$ such that
\begin{enumerate}
\item $\varphi$ has a compact support, i.e. $\varphi(x)=0$ if $|x|_p<p^{-n}$ or $|x|_p>p^n$ for some integer $n$;
\item $\varphi$ is constant on cosets of $A_m$ for some $m\geq1$, i.e. $\varphi(xa)=\varphi(x)$ for all $x\in\mathbb{Q}_p^{\times}$ when $|1-a|_p\leq p^{-m}$.
\end{enumerate}
\end{definition}

The set $S^{\times}$ consists of the functions $\varphi$ such that $\varphi\in S$ and $\varphi(\cdot^{-1})\in S$, where $S=\mathcal{D}(\mathbb{Q}_p)$ is the Bruhat--Schwartz space of set of test functions on the additive group $\mathbb{Q}_p$ (see \cite{GG}).

We will further denote by $L^{\rho}(\mathbb{Q}_p^{\times})$, $L^{\rho}(S_0)$, $\rho\geq1$, the spaces $L^{\rho}$ with respect to the Haar measure $dx^{\times}=\frac{dx}{|x|_p}$.

For $x\in\mathbb{Q}_p^{\times}$ let $x=p^{-n}y$, where $y\in S_0$, $n\in\mathbb{Z}$.

For $\pi\in\widehat{\mathbb{Q}_p^{\times}}$, let $\pi=\theta|\cdot|_p^{i\beta}$, where $\theta\in\widehat{S_0}$, $-\frac{\pi}{\ln p}<\beta\leq\frac{\pi}{\ln p}$.

Then for $f\in L^1(\mathbb{Q}_p,dx^{\times})$,
\begin{equation}\Label{T-4.1}
\int\limits_{\mathbb{Q}_p^{\times}}f(x)dx^{\times}=\int\limits_{\mathbb{Q}_p^{\times}}f(x)\frac{dx}{|x|_p}=\sum\limits_{n=-\infty}^{+\infty}\int\limits_{S_n}f(p^{-n}y)dy.
\end{equation}
For $g\in L^1(\widehat{\mathbb{Q}_p^{\times}},d\pi)$,
\begin{equation}\Label{T-4.2}
\int\limits_{\widehat{\mathbb{Q}_p^{\times}}}g(\pi)d\pi=\sum\limits_{\theta\in \widehat{S_0}}\frac{1}{a}\int\limits_{-\pi/\ln p}^{\pi/\ln p}g(\theta|\cdot|_p^{i\beta})d\beta,
\end{equation}
where $a$ is a positive constant (\cite{T1975}).

Let $f,g\in L^1(\mathbb{Q}_p^{\times})$. The multiplicative convolution of the functions $f$ and $g$ is defined as
\begin{equation}\Label{conv}
(f\ast g)(x)=\int\limits_{\mathbb{Q}_p^{\times}}f(y)g(xy^{-1})\frac{dy}{|y|_p}.
\end{equation}
If $1\leq\rho\leq\infty$, $f\in L^1(\mathbb{Q}_p^{\times})$, $g\in L^{\rho}(\mathbb{Q}_p^{\times})$, then $f\ast g\in L^{\rho}(\mathbb{Q}_p^{\times})$ (see \cite[Pr. 2.40]{Fo}).

For $f\in L^1(\mathbb{Q}_p^{\times})$ \textit{the Mellin transform} of the function $f$ is defined as (see \cite{T1975})
\begin{equation}\Label{2.1}
\left(\mathfrak{M}_{x\to\pi}f\right)(\pi)=\int\limits_{\mathbb{Q}_p^{\times}}f(x)\pi(x)\frac{dx}{|x|_p},\;\pi\in\widehat{\mathbb{Q}_p^{\times}}.
\end{equation}

It follows from \eqref{T-4.1} that
\begin{equation}\Label{2.1}
\left(\mathfrak{M}_{x\to\pi}f\right)(\pi)=\sum\limits_{n=-\infty}^{+\infty}p^{i\beta n}\int\limits_{S_n}f(p^{n}y)\theta(y)dy.
\end{equation}

For $g\in L^1(\widehat{\mathbb{Q}_p^{\times}})$, \textit{the inverse Mellin transform} of the function $g$ is defined as
\begin{equation}\Label{2.2}\
\left(\mathfrak{M}^{-1}_{\pi\to x}g\right)(x)=\int\limits_{\widehat{\mathbb{Q}_p^{\times}}}g(\pi)\pi^{-1}(x)d\pi.
\end{equation}
It follows from \eqref{T-4.2} that \cite{T1975}
\begin{equation}\Label{Mel-1}\
\left(\mathfrak{M}^{-1}_{\pi\to x}g\right)(x)=\frac{1}{a}\sum\limits_{\theta \in \widehat{S_0}}\overline{\theta(p^nx)}\int\limits_{-\pi/\ln p}^{\pi/\ln p}g(\theta|\cdot|^{i\beta})|x|_p^{-i\beta}d\beta,
\end{equation}
where the constant $a$ is determined in such way that the Plancherel identity holds.

The next theorem follows from the theory of the Fourier transform on locally compact Abelian groups (see \cite{Krein} and \cite[Theorem 31.5]{HR2}).
\begin{theorem}\Label{pr-Kr}
If $f\in L_{1,2}(\mathbb{Q}_p^{\times})$, then $\mathfrak{M}f\in L_2(\widehat{\mathbb{Q}_p^{\times}})$ and the Plancherel identity holds:
$$||\mathfrak{M}f||_{\widehat{L_2^{\times}}}=||f||_{L_2^{\times}}.$$
Moreover, extension of the Mellin's transform on $L_2(\mathbb{Q}_p^{\times})$ is a linear isometry $\mathfrak{M}:L_2(\mathbb{Q}_p^{\times})\to L_2(\widehat{\mathbb{Q}_p^{\times}})$.
\end{theorem}

\section{The multiplicative Vladimirov-Taibleson operator}

Let $\alpha>0$. Denote $c_{\alpha}=\dfrac{1-p^{\alpha}}{1-p^{-\alpha-1}}$, and let $B_n^{\times}= B_n\setminus\{0\}$, denote a punctured ball, $n\in\mathbb{Z}$.

For $\varphi\in S^{\times}(\mathbb{Q}_p^{\times})$ define the operator
\begin{equation}\Label{3.1}
(W_0^{\alpha}\varphi)(x)=c_{\alpha}\int\limits_{\mathbb{Q}_p^{\times}}\frac{\varphi(xy^{-1})-\varphi(x)}{|1-y|_p^{\alpha+1}}dy.
\end{equation}

\begin{theorem}\Label{th1}
For $1\leq\rho\leq\infty$, the operator $W_0^{\alpha}:S^{\times}\to L^{\rho}(\mathbb{Q}_p^{\times})$ is a correctly defined linear operator. Moreover,
\begin{equation}\Label{3.2}\
\left(\mathfrak{M}W_0^{\alpha}\varphi\right)(\pi)=c_{\alpha}\left\{\int\limits_{\mathbb{Q}_p^{\times}}\frac{\pi(y)-1}{|1-y|_p^{\alpha+1}}dy\right\}\left(\mathfrak{M}\varphi\right)(\pi).
\end{equation}
\end{theorem}

\begin{proof}
Let $\varphi\in S^{\times}$. Then
\begin{align*}
(W_0^{\alpha}\varphi)(x)&=\int\limits_{B_{-1}^{\times}}\frac{\varphi(xy^{-1})-\varphi(x)}{|1-y|_p^{\alpha+1}}dy+\int\limits_{S_0}\frac{\varphi(xy^{-1})-\varphi(x)}{|1-y|_p^{\alpha+1}}dy+\\
&+\int\limits_{\mathbb{Q}_p^{\times}\setminus B_0}\frac{\varphi(xy^{-1})-\varphi(x)}{|1-y|_p^{\alpha+1}}dy=\\
&=I_1+I_2+I_3.
\end{align*}

We have
$$I_1=\int\limits_{B_{-1}^{\times}}\left(\varphi(xy^{-1})-\varphi(x)\right)dy=\int\limits_{B_{-1}^{\times}}\varphi(xy^{-1})dy-\int\limits_{B_{-1}^{\times}}\varphi(x)dy,$$
$$\int\limits_{B_{-1}^{\times}}\varphi(xy^{-1})dy=\int\limits_{\mathbb{Q}_p^{\times}}\varphi(xy^{-1})\mathbf{1}_{B_{-1}^{\times}}(y)\frac{|y|_p}{|y|_p}dy=(\varphi\ast(\mathbf{1}_{B_{-1}^{\times}}\cdot|\cdot|_p))(x).$$
The function $\mathbf{1}_{B_{-1}^{\times}}\cdot|\cdot|_p$ is in $L^1(\mathbb{Q}_p^{\times})$, since
\begin{align*}
\int\limits_{\mathbb{Q}_p^{\times}}\mathbf{1}_{B_{-1}^{\times}}(y)\cdot|y|_p\frac{dy}{|y|_p}&=\int\limits_{B_{-1}^{\times}}dy=\sum\limits_{n=-\infty}^{-1}\int\limits_{S_n}dy=\\
&=\sum\limits_{n=-\infty}^{-1}(1-p^{-1})p^n=\\
&=(1-p^{-1})\sum\limits_{n=-\infty}^{-1}p^{n}=(1-p^{-1})\frac{p^{-1}}{1-p^{-1}}=p^{-1}.
\end{align*}
Since $\varphi\in L^{\rho}(\mathbb{Q}_p^{\times})$, then $I_1\in L^{\rho}(\mathbb{Q}_p^{\times})$.

We have
$$I_3=\int\limits_{\mathbb{Q}_p^{\times}\setminus B_0}\frac{\varphi(xy^{-1})-\varphi(x)}{|y|_p^{\alpha+1}}dy=\int\limits_{\mathbb{Q}_p^{\times}\setminus B_0}\frac{\varphi(xy^{-1})}{|y|_p^{\alpha+1}}dy-\int\limits_{\mathbb{Q}_p^{\times}\setminus B_0}\frac{\varphi(x)}{|y|_p^{\alpha+1}}dy,$$
$$\int\limits_{\mathbb{Q}_p^{\times}\setminus B_0}\frac{\varphi(xy^{-1})}{|y|_p^{\alpha+1}}dy=\int\limits_{\mathbb{Q}_p^{\times}}\frac{\varphi(xy^{-1})}{|y|_p^{\alpha}}\mathbf{1}_{\mathbb{Q}_p^{\times}\setminus B_0}(y)\frac{dy}{|y|_p}=\Big(\varphi\ast\frac{\mathbf{1}_{\mathbb{Q}_p^*\setminus B_0}}{|\cdot|_p^{\alpha}}\Big)(x).$$
The function $\dfrac{\mathbf{1}_{\mathbb{Q}_p^{\times}\setminus B_0}}{|\cdot|_p^{\alpha}}$ belongs to $L^1(\mathbb{Q}_p^{\times})$, since
$$\int\limits_{\mathbb{Q}_p^{\times}}\frac{\mathbf{1}_{\mathbb{Q}_p^{\times}\setminus B_0}(y)}{|y|_p^{\alpha}}\frac{dy}{|y|_p}=\sum\limits_{n=1}^{\infty}\int\limits_{S_n}\frac{dy}{|y|_p^{\alpha+1}}=(1-p^{-1})\sum\limits_{n=1}^{\infty}p^{-n\alpha}=(1-p^{-1})\frac{p^{-\alpha}}{1-p^{-\alpha}}.$$
Hence, $I_3\in L_{\rho}(\mathbb{Q}_p^{\times})$.

Since $\varphi\in S^{\times}$, there exists such $m\geq1$ that $\varphi(xy^{-1})=\varphi(x)$ when $|1-y^{-1}|_p\leq p^{-m}$. Consider the decomposition of the unit spehere $S_0=\bigcup\limits_{j=0}^{\infty}(A_j\setminus A_{j+1})$. Then $A_j\setminus A_{j+1}=\{x\in\mathbb{Q}_p^{\times}:|1-x|_p=p^{-m}\}$, $j\geq1$, $A_0\setminus A_1\subset S_0(1)$.
\begin{align*}
I_2&=\sum\limits_{j=0}^{\infty}\int\limits_{A_j\setminus A_{j+1}}\frac{\varphi(xy^{-1})-\varphi(x)}{|1-y|_p^{\alpha+1}}dy=\sum\limits_{j=0}^{m-1}\int\limits_{A_j\setminus A_{j+1}}\frac{\varphi(xy^{-1})-\varphi(x)}{|1-y|_p^{\alpha+1}}dy=\\
&=\sum\limits_{j=0}^{m-1}\int\limits_{A_j\setminus A_{j+1}}\frac{\varphi(xy^{-1})-\varphi(x)}{p^{-(\alpha+1)j}}dy=\sum\limits_{j=0}^{m-1}p^{(\alpha+1)j}\int\limits_{A_j\setminus A_{j+1}}\left(\varphi(xy^{-1})-\varphi(x)\right)dy.
\end{align*}
We have
\begin{align*}
\int\limits_{A_j\setminus A_{j+1}}\left(\varphi(xy^{-1})-\varphi(x)\right)dy&=\int\limits_{A_j\setminus A_{j+1}}\varphi(xy^{-1})dy-\varphi(x)\int\limits_{A_j\setminus A_{j+1}}dy=\\
&=(\varphi\ast\mathbf{1}_{A_j\setminus A_{j+1}})(x)-\varphi(x)\int\limits_{A_j\setminus A_{j+1}}dy.
\end{align*}
Since $\varphi\in L_{\rho}(\mathbb{Q}_p^{\times})$, we have $I_2\in L_{\rho}(\mathbb{Q}_p^{\times})$.
Therefore, $W_0^{\alpha}\varphi\in L^{\rho}(\mathbb{Q}_p^{\times})$ and $W_0^{\alpha}:S^{\times}\to L^{\rho}(\mathbb{Q}_p^{\times})$ is correctly defined.
Further, we have
\begin{align*}
\mathfrak{M}[W_0^{\alpha}\varphi](\pi)&=c_{\alpha}\int\limits_{\mathbb{Q}_p^{\times}}[W_0^{\alpha}\varphi](x)\pi(x)\frac{dx}{|x|_p}=c_{\alpha}\int\limits_{\mathbb{Q}_p^{\times}}\left\{\int\limits_{\mathbb{Q}_p^{\times}}\frac{\varphi(xy^{-1})-\varphi(x)}{|1-y|_p^{\alpha+1}}dy\right\}\pi(x)\frac{dx}{|x|_p}=\\
&=c_{\alpha}\int\limits_{\mathbb{Q}_p^{\times}}\int\limits_{\mathbb{Q}_p^*}\frac{\varphi(xy^{-1})}{|1-y|_p^{\alpha+1}}\pi(x)\frac{dx}{|x|_p}dy-c_{\alpha}\int\limits_{\mathbb{Q}_p^{\times}}\int\limits_{\mathbb{Q}_p^{\times}}\frac{\varphi(x)}{|1-y|_p^{\alpha+1}}\pi(x)\frac{dx}{|x|_p}dy.
\end{align*}
After the change of variables $xy^{-1}=z$, $dx=|y|_pdz$ in the inner integral of the first term, we get
\begin{align*}
\mathfrak{M}[W_0^{\alpha}\varphi](\pi)&=c_{\alpha}\int\limits_{\mathbb{Q}_p^{\times}}\int\limits_{\mathbb{Q}_p^{\times}}\frac{\varphi(z)}{|1-y|_p^{\alpha+1}}\pi(yz)\frac{|y|_pdz}{|y|_p|z|_p}\frac{dy}{|y|_p}-c_{\alpha}\int\limits_{\mathbb{Q}_p^{\times}}\frac{1}{|1-y|_p^{\alpha+1}}[\mathfrak{M}\varphi](\pi)dy=\\
&=c_{\alpha}\left\{\int\limits_{\mathbb{Q}_p^{\times}}\frac{\pi(y)-1}{|1-y|_p^{\alpha+1}}dy\right\}[\mathfrak{M}\varphi](\pi).
\end{align*}
\end{proof}

Let us denote
\begin{equation}\Label{3.3}\
K(\pi)=c_{\alpha}\int\limits_{\mathbb{Q}_p^\times}\frac{\pi(y)-1}{|1-y|_p^{\alpha+1}}dy,\;\pi\in\widehat{\mathbb{Q}_p^{\times}},
\end{equation}
where $c_{\alpha}=\dfrac{1-p^{\alpha}}{1-p^{-\alpha-1}}$.
Then the operator $W_0^{\alpha}:S^{\times}\to L^2(\mathbb{Q}_p^{\times})$ is a pseudodifferential operator with the symbol $K$ \eqref{3.3}, namely, it can be written as
\begin{equation}\Label{ghm}
(W_0^{\alpha}\varphi)(x)=\left(\mathfrak{M}^{-1}_{\pi\to x}\left[K\mathfrak{M}_{x\to\pi}\varphi\right]\right)(x),\;\varphi\in S^{\times},\;x\in\mathbb{Q}_p^{\times}.
\end{equation}

Moreover, the integral in \eqref{3.3} is absolutely convergent.

\begin{lemma}\Label{lemma10.1}
Let $\pi\in\widehat{\mathbb Q_p^\times}$ be a multiplicative character. Then
\[
I(\pi):=
\int_{\mathbb Q_p^\times}
\frac{|1-\pi(y)|}{|1-y|_p^{\alpha+1}}\,dy
<\infty,
\qquad \alpha>0.
\]
\end{lemma}

\begin{proof}
We consider separately the cases of characters of positive rank and of rank zero.

\medskip

\noindent {\bf Case 1.} Let us consider character $\pi$ with the rank $k\ge1$. Then $\pi$ is trivial on $A_k$, which means $1-\pi(y)=0$ whenever $|1-y|_p\le p^{-k}.$ Therefore
\[I(\pi) = \int_{|1-y|_p>p^{-k}}\frac{|1-\pi(y)|}{|1-y|_p^{\alpha+1}}\,dy.
\]
Since $|\pi(y)|=1,$ we have $|1-\pi(y)|\le2.$
Therefore
\[I(\pi)\leq 2 \int_{|1-y|_p>p^{-r}}
\frac{dy}{|1-y|_p^{\alpha+1}}.\]
Using the decomposition into spheres centered at $1$,
\[ \{y:\ |1-y|_p>p^{-k}\} = \bigcup_{m=-\infty}^{k-1}
\{y:\ |1-y|_p=p^{-m}\},\]
we obtain
\[ I(\pi) \leq 2(1-p^{-1}) \sum_{m=-\infty}^{k-1}
\frac{p^{-m}}{p^{-m(\alpha+1)}} =
2(1-p^{-1}) \sum_{m=-\infty}^{k-1}
p^{m\alpha}.\]
Since $\alpha>0$, this geometric series converges, therefore $I(\pi)<\infty.$

\medskip

\noindent {\bf Case 2.} Let character $\pi$ has rank \(0\). Then $\pi$ is trivial on $S_0$, so $ 1-\pi(y)=0$ on the whole unit sphere $S_0$. Therefore the singularity at \(y=1\) is completely removed.

Now decompose
\[\mathbb Q_p^\times = \Bigl(\bigcup_{n=-\infty}^{-1}S_n\Bigr) \cup S_0 \cup \Bigl(\bigcup_{n=1}^{\infty}S_n\Bigr), \]
where $S_n=\{y:\ |y|_p=p^n\}.$
For \(n<0\) we have $|1-y|_p=1,$ therefore
\[\int_{S_n} \frac{|1-\pi(y)|}{|1-y|_p^{\alpha+1}}
\,dy \leq 2(1-p^{-1})p^n, \]
and
\[\sum_{n=-\infty}^{-1}p^n<\infty.\]
For \(n>0\) we have $|1-y|_p=p^n,$ and therefore
\[\int_{S_n} \frac{|1-\pi(y)|}{|1-y|_p^{\alpha+1}}
\,dy \leq 2(1-p^{-1}) \frac{p^n}{p^{n(\alpha+1)}}
= 2(1-p^{-1})p^{-n\alpha},\]
and
\[ \sum_{n=1}^{\infty}p^{-n\alpha}<\infty.\]
Therefore $ I(\pi)<\infty$, which gives absolute convergence of integral in \eqref{3.3}.
\end{proof}

\begin{theorem}\Label{dissip}\
Let $\alpha>0$. Then the operator $(-W^\alpha)$ is dissipative on $L^2(\mathbb Q_p^\times)$, i.e.
\begin{equation}\Label{diss}\
\operatorname{Re}\bigl((-W^\alpha)f,f\bigr)\le 0, \quad \text{\rm for any}\ f\in
Dom \,(W^\alpha)
\end{equation}
\end{theorem}

\begin{proof}
By Theorem~\ref{th1}, under the multiplicative Mellin transform, the operator $W^\alpha$ becomes the multiplication operator with symbol \eqref{3.3}.

Since by Lemma \ref{lemma10.1} the integral defining $K(\pi)$ is absolutely convergent, we may pass to the real part under the integral sign:
\[\operatorname{Re}K(\pi) = c_\alpha \int_{\mathbb Q_p^\times} \frac{\operatorname{Re}\pi(y)-1} {|1-y|_p^{\alpha+1}}\,dy. \]
Every character $\pi\in\widehat{\mathbb Q_p^\times}$
is unitary. Therefore $|\pi(y)|=1$ and as any complex number can be represented as
\[\pi (y)=e^{i\varphi(y)}, \quad \text{for any}\ y\in {\mathbb Q_p^\times}.\]
Consequently $\operatorname{Re}\pi(y)\le 1$. Hence $\operatorname{Re}\pi(y)-1\le 0$ for all $y\in\mathbb Q_p^\times$.
Since the denominator $|1-y|_p^{\alpha+1}$ is strictly positive and
$c_\alpha<0$, it follows
\begin{equation}\Label{Re-1}\
\operatorname{Re}K(\pi)\ge 0, \quad \text {for every}\ \ \pi\in\widehat{\mathbb Q_p^\times}.
\end{equation}
Now let $f\in Dom
\,(W^\alpha)$. By the Plancherel formula for the Mellin transform (Theorem \ref{pr-Kr}),
\[(W^\alpha f,f)=\int_{\widehat{\mathbb Q_p^\times}}
K(\pi)\,|\widehat f(\pi)|^2\,d\pi,\]
and therefore
\[\operatorname{Re}(W^\alpha f,f)=\int_{\widehat{\mathbb Q_p^\times}}
\operatorname{Re}K(\pi)\,|\widehat f(\pi)|^2\,d\pi
\ge 0.\]
which implies \eqref{diss} and finishes the proof.
\end{proof}

\begin{lemma}\Label{pr1}
Let $\pi\in\widehat{\mathbb{Q}_p^{\times}}$, $\rank\pi=k\geq1$. Then
\begin{equation}\Label{3.5}
K(\pi)=p^{\alpha k}.
\end{equation}
\end{lemma}

\begin{proof}
We have
$$K(\pi)=c_{\alpha}\sum\limits_{n=-\infty}^{+\infty}\int\limits_{S_n}\frac{\pi(y)-1}{|1-y|_p^{\alpha+1}}dy.$$

Let $n<0$. Then on $S_n$ we have $|y|_p=p^n<1$, so $|1-y|_p=\max\{1,|y|_p\}=1$. Considering the formula \eqref{Vlad-11.40}, we obtain
$$\int\limits_{S_n}\frac{\pi(y)-1}{|1-y|_p^{\alpha+1}}dy=\int\limits_{S_n}\pi(y)dy-\int\limits_{S_n}1dy=-(1-p^{-1})p^n.$$

Let $n=0$. Then it follows from formulas \eqref{Vlad-11.44}, \eqref{Vlad-11.19}
$$\int\limits_{S_0}\frac{\pi(y)-1}{|1-y|_p^{\alpha+1}}dy=\int\limits_{S_0}\frac{\pi(y)}{|1-y|_p^{\alpha+1}}dy-\int\limits_{S_0}\frac{1}{|1-y|_p^{\alpha+1}}dy=\Gamma_p(-\alpha)p^{k\alpha}-\frac{p-2+p^{\alpha}}{p(1-p^{\alpha})}.$$

Let $n>0$. Then on $S_n$ we have $|y|_p=p^n>1$, so $|1-y|_p=|y|_p$, and it follows from \eqref{Vlad-11.40} that
\begin{align*}
\int\limits_{S_n}\frac{\pi(y)-1}{|1-y|_p^{\alpha+1}}dy&
=\int\limits_{S_n}\frac{\pi(y)-1}{|y|_p^{\alpha+1}}dy=\\
&=\frac{1}{p^{n(\alpha+1)}}\left\{\int\limits_{S_n}\pi(y)dy-\int\limits_{S_n}1dy\right\}=-\frac{1}{p^{n\alpha}}(1-p^{-1}).
\end{align*}

Hence
\begin{align*}
K(\pi)&=c_{\alpha}\bigg\{-(1-p^{-1})\sum\limits_{n=-\infty}^{-1}p^n+\Gamma_p(-\alpha)p^{k\alpha}-\frac{p-2+p^{\alpha}}{p(1-p^{\alpha})}-(1-p^{-1})\sum\limits_{n=1}^{+\infty}\frac{1}{p^{\alpha n}}\bigg\}=\\
&=c_{\alpha}\bigg\{-(1-p^{-1})\frac{p^{-1}}{1-p^{-1}}+\Gamma_p(-\alpha)p^{k\alpha}-\frac{p-2+p^{\alpha}}{p(1-p^{\alpha})}-(1-p^{-1})\frac{p^{-\alpha}}{1-p^{-\alpha}}\bigg\}=\\
&=p^{\alpha k}.
\end{align*}
\end{proof}

\begin{proposition}\Label{pr2}
Let the character $\pi\in\widehat{\mathbb{Q}_p^{\times}}$ be such that $\rank\pi=0$, that is $\pi(|x|_p)=|x|_p^{i\beta}$, $x\in\mathbb{Q}_p^{\times}$, for some $\beta\in(-\frac{\pi}{\ln p},\frac{\pi}{\ln p}]$. Then
\begin{align}\Label{3.6}
K(\pi)&=K_0(\beta)=\\
\nonumber
&=c_{\alpha}(1-p^{-1})\bigg\{\frac{p^{-i\beta-1}}{1-p^{-i\beta-1}}-\frac{p^{-1}}{1-p^{-1}}+\frac{p^{i\beta-\alpha}}{1-p^{i\beta-\alpha}}-\frac{p^{-\alpha}}{1-p^{-\alpha}}\bigg\}.
\end{align}
\end{proposition}

\begin{proof}
We have
$$K(\pi)=c_{\alpha}\sum\limits_{n=-\infty}^{+\infty}\int\limits_{S_n}\frac{|y|_p^{i\beta}-1}{|1-y|_p^{\alpha+1}}.$$

When $n<0$,
$$\int\limits_{S_n}\frac{|y|_p^{i\beta}-1}{|1-y|_p^{\alpha+1}}dy=\int\limits_{S_n}\left(|y|_p^{i\beta}-1\right)dy=(p^{ni\beta}-1)(1-p^{-1})p^n.$$
When $n=0$,
$$\int\limits_{S_0}\frac{|y|_p^{i\beta}-1}{|1-y|_p^{\alpha+1}}dy=0.$$
When $n>0$,
$$\int\limits_{S_n}\frac{|y|_p^{i\beta}-1}{|1-y|_p^{\alpha+1}}dy=\int\limits_{S_n}\frac{|y|_p^{i\beta}-1}{|y|_p^{\alpha+1}}dy=\int\limits_{S_n}\frac{p^{n i\beta}-1}{p^{n(\alpha+1)}}dy=\frac{p^{n i\beta}-1}{p^{n\alpha}}(1-p^{-1}).$$
Hence,
\begin{align*}
K(\pi)&=c_{\alpha}(1-p^{-1})\bigg\{\sum\limits_{n=-\infty}^{-1}(p^{ni\beta}-1)p^n+\sum\limits_{n=1}^{+\infty}\frac{p^{ni\beta-1}}{p^{n\alpha}}\bigg\}=\\
&=c_{\alpha}(1-p^{-1})\bigg\{\frac{p^{-i\beta-1}}{1-p^{-i\beta-1}}-\frac{p^{-1}}{1-p^{-1}}+\frac{p^{i\beta-\alpha}}{1-p^{i\beta-\alpha}}-\frac{p^{-\alpha}}{1-p^{-\alpha}}\bigg\}.
\end{align*}
\end{proof}

\section{Eigenfunctions of multiplicative operator}
In this section we find the eigenvalues of the operator $W^{\alpha}$.
\begin{lemma}\Label{lemma1}
The system of functions
\begin{equation}\Label{4.1}
\mathcal{L}=\left\{\varphi(x)=\frac{\theta(p^nx)}{\sqrt{1-p^{-1}}}\mathbf{1}_{S_n}(x),\;x\in \mathbb{Q}_p^{\times} :\theta\in\widehat{S_0},n\in\mathbb{Z}\right\},
\end{equation}
where $\theta$ is a character on the unit sphere $S_0$, forms an orthonormal basis in $L^2(\mathbb{Q}_p^{\times})$.
\end{lemma}

\begin{proof}
It is knows (see \cite{K2020}) that the set of all multiplicative characters $\theta$ on the unit sphere $S_0$ forms an othogonal basis in $L^2(S_0)$. For these characters we can evaluate their norms
$$||\theta||_{L^2(S_0)}^2=\int\limits_{S_0}|\theta(x)|^2\frac{dx}{|x|_p}=\int\limits_{S_0}dx=1-p^{-1}.$$
Therefore, the system of functions $\left\{\dfrac{\theta}{\sqrt{1-p^{-1}}}:\theta\in\widehat{S_0}\right\}$ is an orthonormal basis in $L^2(S_0)$ and the Parseval idenitity holds:
\begin{equation}\Label{pars-S0}
||f||_{L^2(S_0)}^2=\sum\limits_{\theta\in\widehat{S_0}}\abs*{\left(f,\frac{\theta}{\sqrt{1-p^{-1}}}\right)_{L^2(S_0)}}^2,\;f\in L^2(S_0).
\end{equation}

We will show that the system of functions $\left\{\dfrac{\theta(p^n\cdot)}{\sqrt{1-p^{-1}}}:\theta\in\widehat{S_0}\right\}$ is an orthonormal basis in $L^2(S_n)$, $n\in\mathbb{Z}$. Indeed, for $\theta\in\widehat{S_0}$ we have
$$\norm{\frac{\theta(p^n\cdot)}{\sqrt{1-p^{-1}}}}_{L^2(S_n)}^2=\int\limits_{S_n}\abs*{\frac{\theta(p^n\cdot)}{\sqrt{1-p^{-1}}}}^2\frac{dx}{|x|_p}=\frac{1}{1-p^{-1}}\int\limits_{S_n}\frac{dx}{p^n}=1,$$
and also for $\theta_1,\theta_2\in\widehat{S_0}$, $\theta_1\neq\theta_2$,
\begin{align*}
\left(\frac{\theta_1(p^n\cdot)}{\sqrt{1-p^{-1}}},\frac{\theta_2(p^n\cdot)}{\sqrt{1-p^{-1}}}\right)_{L^2(S_n)}&=\int\limits_{S_n}\frac{\theta_1(p^nx)\overline{\theta_2(p^nx)}}{1-p^{-1}}\frac{dx}{|x|_p}=\\
&=\frac{1}{(1-p^{-1})p^n}\int\limits_{S_n}\theta_1(p^nx)\overline{\theta_2(p^nx)}dx.
\end{align*}
After the substitution $y=p^nx$, $dx=p^ndy$, it follows from formula \eqref{Vlad-11.40} that
$$\left(\frac{\theta_1(p^n\cdot)}{\sqrt{1-p^{-1}}},\frac{\theta_2(p^n\cdot)}{\sqrt{1-p^{-1}}}\right)_{L^2(S_n)}=\frac{1}{(1-p^{-1})p^n}\int\limits_{S_0}\theta_1(y)\overline{\theta_2(y)}dy=0.$$
Moreover, for any function $f\in L^2(S_n)$ we get
$$||f||_{L^2(S_n)}^2=\int\limits_{S_n}|f(x)|^2\frac{dx}{|x|_p}=\int\limits_{S_0}|f(p^{-n}y)|^2dy.$$

Let $g(y)=f(p^{-n}y)$, $y\in S_0$. Then $g\in L^2(S_0)$ and using \eqref{pars-S0}, we obtain
\begin{align*}
||f||_{L^2(S_n)}^2=||g||_{L^2(S_0)}^2&=\sum\limits_{\theta\in\widehat{S_0}}\abs*{\left(g,\frac{\theta}{\sqrt{1-p^{-1}}}\right)_{L^2(S_0)}}^2=\\
&=\sum\limits_{\theta\in\widehat{S_0}}\abs*{\left(f(p^{-n}\cdot),\frac{\theta}{\sqrt{1-p^{-1}}}\right)_{L^2(S_0)}}^2.
\end{align*}
We have
\begin{align*}
\left(f(p^{-n}\cdot),\frac{\theta}{\sqrt{1-p^{-1}}}\right)_{L^2(S_0)}&=\frac{1}{\sqrt{1-p^{-1}}}\int\limits_{S_0}f(p^{-n}y)\overline{\theta(y)}\frac{dy}{|y|_p}=\\
&=\frac{1}{\sqrt{1-p^{-1}}}\int\limits_{S_n}f(x)\overline{\theta(p^nx)}\frac{dx}{|x|_p}=\\
&=\left(f,\frac{\theta(p^n\cdot)}{\sqrt{1-p^{-1}}}\right)_{L^2(S_n)}.
\end{align*}
Hence,
\begin{equation}\Label{pars-Sn}\
||f||_{L^2(S_n)}^2=\sum\limits_{\theta\in\widehat{S_0}}\abs*{\left(f,\frac{\theta(p^n\cdot)}{\sqrt{1-p^{-1}}}\right)_{L^2(S_n)}}^2,
\end{equation}
so the Parseval identity holds.

Now consider $f\in L^2(\mathbb{Q}_p^{\times})$. The system $\mathcal{L}$ is orthonormal in $L^2(\mathbb{Q}_p^{\times})$. Moreover,
$$||f||_{L^2(\mathbb{Q}_p^{\times})}^2=\int\limits_{\mathbb{Q}_p^{\times}}|f(x)|^2\frac{dx}{|x|_p}=\sum\limits_{n=-\infty}^{+\infty}\int\limits_{S_n}|f(x)|^2\frac{dx}{|x|_p}=\sum\limits_{n=-\infty}^{+\infty}\int\limits_{S_n}||f_n||_{L^2(S_n)}^2,$$
where $f_n=f\restriction_{S_n}$, $n\in\mathbb{Z}$. It follows from \eqref{pars-Sn} that
\begin{align*}
||f||_{L^2(\mathbb{Q}_p^{\times})}^2&=\sum\limits_{n=-\infty}^{+\infty}\sum\limits_{\theta\in\widehat{S_0}}\abs*{\left(f_n,\frac{\theta(p^n\cdot)}{\sqrt{1-p^{-1}}}\right)_{L^2(S_n)}}^2=\\
&=\sum\limits_{n=-\infty}^{+\infty}\sum\limits_{\theta\in\widehat{S_0}}\abs*{\left(f,\frac{\theta(p^n\cdot)}{\sqrt{1-p^{-1}}\mathbf{1}_{S_n}}\right)_{L^2(\mathbb{Q}_p^{\times})}}^2.
\end{align*}
Since for $\mathcal{L}$ the Parseval identity hold, it follows that $\mathcal{L}$ is a basis in $L^2(\mathbb{Q}_p^{\times})$.
\end{proof}

\begin{lemma}\Label{lemma2}
The set of test functions $S^{\times}$ is dense in $L^2(\mathbb{Q}_p^{\times})$.
\end{lemma}

\begin{proof}
We will show that $\mathcal{L}\subset S^{\times}$. Indeed, let $\varphi\in\mathcal{L}$. Then $\supp\varphi=S_n$. If $\rank(\theta)=k\geq1$, then for all $a\in A_k$ we have $a\in S_0$, so $|xa|_p=|x|_p|a|_p=|x|_p$. Then the condition $\mathbf{1}_{S_n}(xa)=1$ is equivalent to the condition $\mathbf{1}_{S_n}(x)=1$, and we have for all $x\in\mathbb{Q}_p^{\times}$, $a\in A_k$:
$$\varphi(xa)=\theta(p^nxa)\mathbf{1}_{S_n}(xa)=\varphi(x).$$
Note that $S^{\times}$ is a linear set, that it for all $\alpha,\beta\in\mathbb{C}$, $\varphi,\psi\in S^{\times}$ it holds that  $\alpha\varphi+\beta\psi\in S^{\times}$. Hence, $span(\mathcal{L})\subset S^{\times}$. Since $\mathcal{L}$ is a basis in $L^2(\mathbb{Q}_p^{\times})$, it follows that $span(\mathcal{L})$, qnd therefore $S^{\times}$, is dense in $L^2(\mathbb{Q}_p^{\times})$.
\end{proof}

\begin{lemma}\Label{lemma4}
If $\varphi\in\mathcal{L}$, where the set $\mathcal{L}$ is defined in \eqref{4.1}, $\rank \theta=k\geq1$, then $\varphi$ is an eigenfunction of the operator $W_0^{\alpha}$ with the eigenvalue $\lambda_k=p^{\alpha k}$, i.e.
\begin{equation}\Label{eigen}
(W_0^{\alpha}\varphi)(x)=p^{\alpha k}\varphi(x),\;x\in\mathbb{Q}_p^{\times}.
\end{equation}
\end{lemma}

\begin{proof}
Let $\rank \theta=k\geq1$. Then for $N\in\mathbb{Z}$ and $\varphi\in\mathcal{L}$ we have
\begin{align*}
(W_0^{\alpha}\varphi)(x)&=
\frac{c_{\alpha}}{\sqrt{1-p^{-1}}}\int\limits_{\mathbb{Q}_p^{\times}}\frac{\theta(p^Nxy^{-1})\mathbf{1}_{S_N}(xy^{-1})-\theta(p^Nx)\mathbf{1}_{S_N}(x)}{|1-y|_p^{\alpha+1}}dy=\\
&=\frac{c_{\alpha}}{\sqrt{1-p^{-1}}}\bigg\{\int\limits_{\mathbb{Q}_p^{\times}}\frac{\theta(p^Nxy^{-1})\mathbf{1}_{S_N}(xy^{-1})}{|1-y|_p^{\alpha+1}}dy-\theta(p^Nx)\mathbf{1}_{S_N}\int\limits_{\mathbb{Q}_p^{\times}}\frac{dy}{|1-y|_p^{\alpha+1}}\bigg\}=\\
&=\frac{c_{\alpha}}{\sqrt{1-p^{-1}}}\{I_1-I_2\}.
\end{align*}

Let $|x|_p=p^n$, $n\in\mathbb{Z}$. We have $xy^{-1}\in S_N$ if and only if  $|y|_p=p^{-N}|x|_p$, that is $|y|_p=p^{-N+n}$. Hence,
$$I_1=\int\limits_{S_{-N+n}}\frac{\theta(p^Nxy^{-1})}{|1-y|_p^{\alpha+1}}dy.$$
When $N>n$, considering the formula \eqref{Vlad-11.40}, we obtain
\begin{align*}
I_1&=\int\limits_{S_{-N+n}}\frac{\theta(p^{n+N-n}xy^{-1})}{|1-y|_p^{\alpha+1}}dy=\theta(p^nx)\int\limits_{S_{-N+n}}\frac{\overline{\theta(y)}}{|1-y|_p^{\alpha+1}}dy=\\
&=\theta(p^nx)\int\limits_{S_{-N+n}}\overline{\theta(p^{-N+n}y)}dy=0.
\end{align*}
When $N=n$ it follows from \eqref{Vlad-11.44} that
$$I_1=\int\limits_{S_0}\frac{\theta(p^nxy^{-1})}{|1-y|_p^{\alpha+1}}dy=\theta(p^nx)\int\limits_{S_{-N+n}}\frac{\overline{\theta(y)}}{|1-y|_p^{\alpha+1}}dy=\theta(p^nx)\Gamma_p(-\alpha)p^{\alpha k}.$$
When $n<N$ we have from \eqref{Vlad-11.40},
$$I_1=\int\limits_{S_{-N+n}}\frac{\theta(p^{n+N-n}xy^{-1})}{|1-y|_p^{\alpha+1}}dy=\theta(p^nx)\int\limits_{S_{-N+n}}\frac{\overline{\theta({p^{-N+n}y})}}{|1-y|_p^{\alpha+1}}dy=0.$$
So,
$$I_1=\mathbf{1}_{S_N}(x)\theta(p^Nx)\Gamma_p(-\alpha).$$
It follows from \eqref{Vlad-11.3}, \eqref{Vlad-11.19} that
\begin{align*}
I_2&=\int\limits_{\mathbb{Q}_p^{\times}}\frac{dy}{|1-y|_p^{\alpha+1}}=\sum\limits_{n=-\infty}^{-1}\int\limits_{S_n}dy+\int\limits_{S_0}\frac{dy}{|1-y|_p^{\alpha+1}}dy+\sum\limits_{n=1}^{+\infty}\int\limits_{S_n}\frac{dy}{|y|_p^{\alpha+1}}dy=\\
&=\sum\limits_{n=-\infty}^{-1}(1-p^{-1})p^n+\frac{p-2+p^{\alpha}}{p(1-p^{\alpha})}+\sum\limits_{n=1}^{+\infty}(1-p^{-1})p^{-n\alpha}=\\
&=(1-p^{-1})\frac{p^{-1}}{1-p^{-1}}+\frac{p-2+p^{\alpha}}{p(1-p^{\alpha})}+(1-p^{-1})\frac{p^{-\alpha}}{1-p^{-\alpha}}.
\end{align*}
Therefore,
\begin{multline}
(W_0^{\alpha}\varphi)(x)=\frac{c_{\alpha}}{\sqrt{1-p^{-1}}}\mathbf{1}_{S_N}(x)\theta(p^{N}x)\times\\
\times \left\{\Gamma_p(-\alpha)+p^{-1}+\frac{p-2+p^{\alpha}}{p(1-p^{\alpha})}+(1-p^{-1})\frac{p^{-\alpha}}{1-p^{-\alpha}}\right\}=p^{\alpha k}\varphi(x).
\end{multline}
\end{proof}

\begin{remark}
The obtained system of eigenfunctions of the operator $W_0^{\alpha}$ does not form a basis in $L^2(\mathbb{Q}_p^{\times})$, since when $\theta=1$, we have $\rank\theta=0$, and the functions of the form $\varphi=\mathbf{1}_{S_n}$, $n\in\mathbb{Z}$, are not eigenfunctions of $W^{\alpha}$.
\end{remark}


\begin{remark}\Label{lemma3}\rm\
Note that the functions $\psi(x)=\pi(x)\mathbf{1}_{R_N}(x)$, $x\in\mathbb{Q}_p^{\times}$, $N\geq1$, where $\pi\in\widehat{\mathbb{Q}_p^{\times}}$ is a fixed character, $\rank\pi=k\geq1$, also form a set of eigenfunctions of the operator $W_0^{\alpha}$ with eigenvalues $p^{\alpha k}$. The proof is similar to Lemma \ref{lemma4}.
\end{remark}

\section{Properties of the multiplicative operator}\Label{Sec5}\
In this section we investigate the properties of the closure of the operator $W_0^{\alpha}$.

Consider the operator $W_0^{\alpha}$ as a linear unbounded operator in $L^2(\mathbb{Q}_p^{\times})$ with the domain $S^{\times}$,
$$W_0^{\alpha}\varphi=\mathfrak{M}^{-1}K\mathfrak{M}\varphi,\;\varphi\in S^{\times}.$$
For all $\varphi\in S^{\times}$ we have:
$$(W_0^{\alpha}\varphi,\psi)=(\varphi,\mathfrak{M}^{-1}\overline{K}\mathfrak{M}\psi),$$
so the condition
$$(W_0^{\alpha}\varphi,\psi)=(\varphi,\eta)$$
for $\eta\in L^2(\widehat{\mathbb{Q}_p^{\times}})$ is equivalent to the condition
$$(\varphi,\mathfrak{M}^{-1}\overline{K}\mathfrak{M}\psi)=(\varphi,\eta),$$
if $\mathfrak{M}^{-1}\bar{K}\mathfrak{M}\psi\in L^2(\mathbb{Q}_p^{\times})$, that is $\overline{K}\mathfrak{M}\psi\in L^2(\widehat{\mathbb{Q}_p^{\times}})$. Since by Lemma \ref{lemma2} $S^{\times}$ is dense in $L^2(\mathbb{Q}_p^{\times})$, it follows that $\eta=\mathfrak{M}^{-1}\bar{K}\mathfrak{M}\psi$, so we have the adjoint operator
$$(W_0^{\alpha})^{\ast}\psi=\mathfrak{M}^{-1}\overline{K}\mathfrak{M}\psi,$$
with the domain
$$D\left((W_0^{\alpha})^{\ast}\right)=\left\{\psi\in L^2(\mathbb{Q}_p^{\times})|\overline{K}\mathfrak{M}\psi\in L^2(\widehat{\mathbb{Q}_p^{\times}})\right\}.$$
Since $(W_0^{\alpha})^{\ast}$ is closed, the extension of the operator $W_0^{\alpha}$, which we will now denote as $W^{\alpha}$ with the domain
$$D(W^{\alpha})=\left\{\varphi\in L^2(\mathbb{Q}_p^{\times})|K\mathfrak{M}\varphi\in L^2(\widehat{\mathbb{Q}_p^{\times}})\right\}$$
is also a closed linear operator.

We will show that radial functions are in the domain of $W^{\alpha}$.

\begin{proposition}\Label{pr5.1}
Let $\varphi\in L^1(\mathbb{Q}_p^{\times})$ be a radial function, that is $\varphi(x)=\varphi(|x|_p)$, $x\in\mathbb{Q}_p^{\times}$. Then $\varphi\in D(W^{\alpha})$.
\end{proposition}

\begin{proof}
We have
$$\int\limits_{\mathbb{Q}_p^{\times}}|\varphi(|x|_p)|\frac{dx}{|x|_p}=\sum\limits_{n=-\infty}^{+\infty}\int\limits_{S_n}|\varphi(p^n)|\frac{dx}{p^n}=(1-p^{-1})\sum\limits_{n=-\infty}^{+\infty}|\varphi(p^n)|.$$
So the condition $\varphi\in L^1(\mathbb{Q}_p^{\times})$ is equivalent to the condition $\sum\limits_{n=-\infty}^{+\infty}|\varphi(p^n)|<\infty$.

It follows from \eqref{2.1}, \eqref{T-4.2} that
\begin{align*}
\int\limits_{\widehat{\mathbb{Q}_p^{\times}}}\abs*{K(\pi)(\mathfrak{M}\varphi)(\pi)}^2d\pi
&=\int\limits_{\widehat{\mathbb{Q}_p^{\times}}}\abs*{K(\pi)\int\limits_{\mathbb{Q}_p^{\times}}\varphi(x)\pi(x)\frac{dx}{|x|_p}}^2d\pi=\\
&=\int\limits_{\widehat{\mathbb{Q}_p^{\times}}}\abs*{K(\pi)\sum\limits_{n=-\infty}^{+\infty}\int\limits_{S_n}\varphi(x)\theta(x)p^{i\beta n}\frac{dx}{|x|_p}}^2d\pi=\\
&=\frac{1}{a}\sum\limits_{\theta\in\widehat{S_0}}\int\limits_{-\pi/\ln p}^{\pi/\ln p}\abs*{K(\theta|\cdot|_p^{i\beta})}^2\abs*{\sum\limits_{n=-\infty}^{+\infty}p^{i\beta n}\int\limits_{S_n}\frac{\varphi(x)\theta(x)\,dx}{|x|_p}}^2d\beta=\\
&=\frac{1}{a}\sum\limits_{\theta\in\widehat{S_0},\theta\neq1}\int\limits_{-\pi/\ln p}^{\pi/\ln p}p^{2\alpha\deg(\theta)}\abs*{\sum\limits_{n=-\infty}^{+\infty}p^{i\beta n}\varphi(p^n)\int\limits_{S_n}\frac{\theta(x)\, dx}{|x|_p}}^2d\beta+\\
&+\frac{1}{a}\int\limits_{-\pi/\ln p}^{\pi/\ln p}|K_0(\beta)|^2\abs*{\sum\limits_{n=-\infty}^{+\infty}p^{i\beta n}\varphi(p^n)\int\limits_{S_n}\frac{dx}{|x|_p}}^2d\beta=I+II.
\end{align*}
Since in the first sum $\rank\theta>0$ for all $\theta\in\widehat{S_0}$, $\theta\neq1$, it follows from \eqref{Vlad-11.40} that all the inner integrals are zero. Hence, $I=0$. We also have
\begin{align*}
II&\leq\frac{1-p^{-1}}{a}\int\limits_{-\pi/\ln p}^{\pi/\ln p}|K_0(\beta)|^2\sum\limits_{n=-\infty}^{+\infty}|p^{i\beta n}\varphi(p^n)|d\beta=\\
&=\frac{1-p^{-1}}{a}\int\limits_{-\pi/\ln p}^{\pi/\ln p}|K_0(\beta)|^2\sum\limits_{n=-\infty}^{+\infty}|\varphi(p^n)|d\beta.
\end{align*}
Since the sum $\sum\limits_{n=-\infty}^{+\infty}|\varphi(p^n)|$ is finite and the function $K_0$ is bounded on $[-\frac{\pi}{\ln p},\frac{\pi}{\ln p}]$, then $II<\infty$.
So, $K(\cdot)\mathfrak{M}\varphi\in L^2(\widehat{\mathbb{Q}_p^{\times}})$.
\end{proof}


Note that the operator $W^{\alpha}$ is not symmetric, but  only dissipative due to the Theorem~\ref{dissip}. Thus, to construct the semigroup associated with the operator $(-W^{\alpha})$, we need to check its maximal dissipativity.
\begin{theorem}\Label{SemTt}\
The operator $(-W^\alpha)$ is the generator of a contraction $C_0$--semigroup $(T_t)_{t\ge0}$
on $L^2(\mathbb Q_p^\times)$.
\end{theorem}

\begin{proof}
By Theorem~\ref{dissip}, the operator $(-W^{\alpha})$ is dissipative. To prove maximal dissipativity, it suffices to show that
$$\Ran(\lambda I+W^{\alpha})=L^2(\mathbb{Q}_p^{\times})$$
for some $\lambda\geq0$, which will be chosen later.

Let $g\in L^2(\mathbb{Q}_p^{\times})$. Consider the equation
$$(\lambda I+W^{\alpha})f=g.$$
Applying the Mellin transform and using that
$$(\mathfrak{M}\,W^{\alpha}f)(\pi)=K(\pi)(\mathfrak{M}f)(\pi),\;\pi\in\mathbb{Q}_p^{\times},$$
we obtain
\begin{equation}\Label{A}
(\lambda+K(\pi))(\mathfrak{M}f)(\pi)=(\mathfrak{M})g(\pi).
\end{equation}

It follows from \eqref{3.6} that the function $\real K_0$ is continuous and hence reaches a maximum at $[-\frac{\pi}{\ln p},\frac{\pi}{\ln p}]$, which by \eqref{Re-1} is a positive constant:
$$\max\limits_{(-\frac{\pi}{\ln p},\frac{\pi}{\ln p}]}\real[K_0(\beta)]=d_{\alpha}>0.$$
Hence, for $\lambda>d_{\alpha}$, $\lambda\neq p^{\alpha k}$, $k\geq1$,
$$|\lambda+K(\pi)|\geq\min\limits_{k\geq1}(|\lambda-d_{\alpha}|,|\lambda-p^{\alpha k}|)\eqqcolon \delta>0.$$

Therefore from \eqref{A} we have
$$(\mathfrak{M}f)(\pi)=\frac{(\mathfrak{M}g)(\pi)}{\lambda+K(\pi)}$$
and $(\mathfrak{M}f)(\pi)$ is well defined for every character $\pi$. Moreover,
$$\abs*{\frac{1}{\lambda+K(\pi)}}\leq\frac{1}{\delta},$$
hence $\mathfrak{M}f\in L^2(\widehat{\mathbb{Q}_p^{\times}}).$ Indeed, for $g\in L^2(\widehat{\mathbb Q_p^\times})$,
$$\int\limits_{\widehat{\mathbb{Q}_p^{\times}}}\abs*{(\mathfrak{M}f)(\pi)}^2 d\pi=\int\limits_{\widehat{\mathbb{Q}_p^{\times}}}\abs*{\frac{(\mathfrak{M}g)(\pi)}
{\lambda+K(\pi)}}^2 d\pi\leq \dfrac{1}{\delta^2}\int\limits_{\widehat{\mathbb{Q}_p^{\times}}}\abs*{(\mathfrak{M}g)(\pi)}^2 d\pi,
$$
which is finite due to Theorem~\ref{pr-Kr}, we have that $\mathfrak{M}g\in L^2(\widehat{\mathbb{Q}_p^{\times}})$ for $g\in L^2(\widehat{\mathbb{Q}_p^{\times}})$.

Further, by the surjectivity of the Mellin transform, there exists a unique $f\in L^2(\mathbb Q_p^\times)$ such that
$$(\mathfrak{M}f)(\pi)=\frac{(\mathfrak{M}g)(\pi)}{\lambda+K(\pi)}.$$
Hence
$$(\lambda+K(\pi))(\mathfrak{M}f)(\pi)=(\mathfrak{M}g)(\pi),$$
or equivalently,
$$\mathfrak{M}\left[(\lambda I+W^\alpha)f\right]=\mathfrak{M}g.$$
Since the Mellin transform is injective, we conclude that
$$(\lambda I+W^\alpha)f=g.$$

Thus $\Ran(\lambda I+W^{\alpha})=L^2(\mathbb{Q}_p^{\times}),$ and $(-W^{\alpha})$ is maximally dissipative. Therefore Theorem~3.3 in \cite{G} now implies that $(-W^{\alpha})$ generates a contraction $(C_0)$--semigroup $(T_t)_{t\ge0}$ on $L^2(\mathbb{Q}_p^{\times})$.
\end{proof}

\begin{corollary}\Label{RepTt}\
The semigroup $(T_t)_{t\ge0}$ generated by $(-W^\alpha)$ satisfies
\begin{equation}\Label{Tt1}\
\mathfrak M(T_tf)(\pi)=e^{-tK(\pi)}\mathfrak M f(\pi),
\qquad
f\in L^2(Q_p^\times).
\end{equation}
\end{corollary}
\begin{proof} Representation \eqref{Tt1} follows from \eqref{ghm} and unitarity of Mellin transform.
\end{proof}

\begin{proof}
By Theorem~\ref{th1}, under the multiplicative Mellin transform the operator $W^\alpha$ becomes the multiplication operator with symbol $K$.
Since the Mellin transform is unitary, the semigroup generated by $(-W^\alpha)$ is unitarily equivalent to the semigroup generated by the multiplication operator to the function
$(-K)$, which is the multiplication semigroup by $e^{-tK}$. This implies \eqref{Tt1}.
\end{proof}

\section{Fundamental solution of the Cauchy problem}
Let us investigate properties of the fundamental solution of the Cauchy problem
\begin{align}\Label{6.1}
\frac{\partial}{\partial t}u(t,x)+W^{\alpha}u(t,x)&=0,\;(t,x)\in \mathbb{R}_{+}\times\mathbb{Q}_p^{\times},\\
\Label{6.2}
u(0,x)&=u_0(x),
\end{align}
where the function $u_0\in D(W^{\alpha})$.

\begin{definition}
We say that the function $u:[0,+\infty)\times\mathbb{Q}_p^{\times}\to\mathbb{C}$ is a solution of the problem \eqref{6.1}--\eqref{6.2} if $u\in C\left([0,+\infty),D(W^{\alpha})\right)\cap C^1\left([0,+\infty),L^2(\mathbb{Q}_p^{\times})\right)$ and $u$ satisfies \eqref{6.1} for all $t\geq0$.
\end{definition}

Applying the Mellin transform to \eqref{6.1}, we get
$$\mathfrak{M}_{x\to\pi}\left[\frac{\partial}{\partial t}u(t,x)\right]+\MMxp[W^{\alpha}u(t,x)]=0.$$
Denote $\widehat{u}=\mathfrak{M}_{x\to\pi}u$, $\widehat{u_0}=\mathfrak{M}_{x\to\pi}u_0$. We have the equation
\begin{equation}\Label{6.3}
\frac{\partial}{\partial t}\widehat{u}(t,\pi)+K(\pi)\widehat{u}(t,\pi)=0,\;(t,\pi)\in \mathbb{R}_{+}\times\widehat{\mathbb{Q}_p^{\times}}.
\end{equation}
For any fixed $\pi$, the general solution to the ordinary differential equation \eqref{6.3} w.r.t. the variable $t$ is the function
$$\widehat{u}(t,\pi)=ce^{-tK(\pi)},\;(t,\pi)\in \mathbb{R}_{+}\times\widehat{\mathbb{Q}_p^{\times}}.$$
From the initial condition $\eqref{6.2}$ we get
$$\widehat{u}(t,\pi)=\widehat{u_0}(\pi)e^{-tK(\pi)},\;(t,\pi)\in \mathbb{R}_{+}\times\widehat{\mathbb{Q}_p^{\times}}.$$
Applying the inverse Mellin transform, we get
\begin{equation}\Label{6-sol}
u(t,x)=\mathfrak{M}_{\pi\to x}^{-1}\left[\widehat{u_0}(\pi)e^{-tK(\pi)}\right]=u_0(x)\ast\MMpx[e^{-tK(\pi)}](t,x),
\end{equation}
where $\ast$ denotes the multiplicative convolution defined in \eqref{conv}.

We will call the function
\begin{equation}\Label{6.4}
Z(t,x)=\mathfrak{M}_{\pi\to x}^{-1} [e^{-tK(\pi)}],\;(t,x)\in \mathbb{R}_{+}\times\mathbb{Q}_p^{\times},
\end{equation}
the \textit{fundamental solution} of the problem \eqref{6.1}-\eqref{6.2}.

\begin{theorem}\label{th6.1}
The fundamental solution $Z$ of the problem \eqref{6.1}-\eqref{6.2} satisfies the following estimate for all $t>0$ and $x\in\mathbb{Q}_p^{\times}$:
\begin{equation}\Label{est-1}
|Z(t,x)|\leq C\left(t^{-\frac{1}{\alpha}}+1\right),
\end{equation}
where $C>0$ is independent of $t$ and $x$.
Moreover, for $|x|_p\neq1$,
\begin{equation}\Label{est-2}
|Z(t,x)|\leq C\min\left\{1,\frac{t}{\abs*{\ln|x|_p}}\right\}.
\end{equation}
\end{theorem}

\begin{proof}
Using representation \eqref{Mel-1} for inverse Mellin transfrom, we get for $|x|_p=p^n$, $n\in\mathbb{Z}$,
\begin{align*}
Z(t,x)&=\mathfrak{M}_{\pi\to x}^{-1}[e^{-tK(\pi)}] =\int\limits_{\widehat{\mathbb{Q}_p^{\times}}}e^{-tK(\pi)}\pi^{-1}(x)d\pi= \\
&=\frac{1}{a}\sum\limits_{\theta \in \widehat{S_0}}\overline{\theta(p^nx)}
\int\limits_{-\pi/\ln p}^{\pi/\ln p}\exp\big(-tK(\theta|\cdot|^{i\beta})\big)|x|_p^{-i\beta}d\beta=\\
&=\frac{1}{a}\sum\limits_{\theta\in\widehat{S_0},\theta\neq1}\overline{\theta(p^nx)}\int\limits_{-\pi/\ln p}^{\pi/\ln p}\exp\big({-tp^{\alpha\,\rank(\theta^{-1}|\cdot|^{-i\beta})}}\big)|x|_p^{-i\beta}d\beta+\\
&+\dfrac{1}{a}\int\limits_{-\pi/\ln p}^{\pi/\ln p}e^{-K_0(\beta)t}|x|_p^{-i\beta}d\beta=I_1+I_2.
\end{align*}

For $|x|_p=p^n$, $n\neq0$, we have
\begin{align*}
\int\limits_{-\pi/\ln p}^{\pi/\ln p}|x|_p^{-i\beta}d\beta&=\int\limits_{-\pi/\ln p}^{\pi/\ln p}e^{-i\beta\ln|x|_p}d\beta=\\
&=\int\limits_{-\pi/\ln p}^{\pi/\ln p}\left[\cos(-\beta\ln|x|_p)+i\sin(-\beta\ln|x|_p)\right]d\beta=\\
&=\frac{2\sin(\frac{\pi}{\ln p}\ln|x|_p)}{\ln|x|_p}=\frac{2\sin(\pi n)}{n\ln p}=0,
\end{align*}
and for $n=0$, i.e. $|x|_p=1$,
$$\int\limits_{-\pi/\ln p}^{\pi/\ln p}|x|_p^{-i\beta}d\beta=\int\limits_{-\pi/\ln p}^{\pi/\ln p}d\beta=\frac{2\pi}{\ln p}.$$
Hence,
\begin{equation}\Label{idt}
I_1=0,\;|x|_p\neq1,
\end{equation}
and for $|x|_p=1$,
$$I_1=\frac{1}{a}\sum\limits_{\theta\in\widehat{S_0},\theta\neq1}\overline{\theta(p^nx)}e^{-\alpha
\,\rank\theta\, t}\int\limits_{-\pi/\ln p}^{\pi/\ln p}|x|_p^{-i\beta}d\beta
=\frac{2\pi}{a\ln p}
\sum\limits_{\theta\in\widehat{S_0},\theta\neq1}\overline{\theta(p^nx)}e^{-\alpha\,\rank\theta\, t}.$$

Since on the unit sphere there are exactly $p-2$ characters of rank $1$ and $(p-1)^2p^{k-2}$ characters of rank $k\geq2$ (see e.g.,\cite[Section IV.B]{R}), we will enumerate them by $\theta_{k,m}$, where $k=\rank\theta$. We have
$$
\sum\limits_{\theta\in\widehat{S_0}}\overline{\theta(p^nx)}e^{-tp^{\alpha\rank\theta}}=\sum\limits_{m=1}^{p-2}\overline{\theta_{1,m}(p^nx)}e^{-p^{\alpha} t}+\sum\limits_{k=2}^{+\infty}\sum\limits_{m=1}^{(p-1)^2p^{k-2}}\overline{\theta_{k,m}(p^nx)}e^{-p^{\alpha k}t}.$$

For a finite Abelian group $G$ there is a known formula \cite[Th. 4.1]{Co}
\begin{equation}\Label{Co1}
\sum\limits_{\chi\in\widehat{G}}\chi(g)=\begin{cases}
|G|,&\text{ if }g=1,\\
0,&\text{ if }g\neq1.
\end{cases}
\end{equation}

Denote for $k\geq0$: $\mathfrak{A}_k=\{\theta\in\widehat{S_0}|\deg\theta\leq k\}$, $\mathfrak{A}_0=\{1\}$ (since on $S_0$ there is exactly one character of rank $0$), $\mathfrak{S}_k=\{\theta\in\widehat{S_0}|\deg\theta= k\}$. Then $\mathfrak{A}_k$ is a finite subgroup in $\widehat{S_0}$. Using \eqref{Co1}, we get for $k\geq1$
$$\sum\limits_{\theta\in\mathfrak{A}_k}\theta(x)=\begin{cases}
|\mathfrak{A}_k|,&\text{ if }x\in A_k,\\
0,&\text{ if }x\notin A_k,
\end{cases}$$
and for $k=0$,
$$\sum\limits_{\theta\in\mathfrak{A}_0}\theta(x)=|\mathfrak{A}_0|=1.$$
Since $\mathfrak{S}_k=\mathfrak{A}_k\setminus\mathfrak{A}_{k-1}$, we have for $k\geq2$
$$\sum\limits_{\theta\in\mathfrak{S}_k}\theta(x)=\sum\limits_{\theta\in\mathfrak{A}_k}\theta(x)-\sum\limits_{\theta\in\mathfrak{A}_{k-1}}\theta(x)=\begin{cases}
0,&\text{ if }x\in S_0\setminus A_{k-1},\\
-|\mathfrak{A}_{k-1}|,&\text{ if }x\in A_{k-1}\setminus A_k,\\
|\mathfrak{A}_k|-|\mathfrak{A}_{k-1}|,&\text{ if }x\in A_k.
\end{cases}$$
Here
$$|\mathfrak{A}_k|-|\mathfrak{A}_{k-1}|=|\mathfrak{S}_k|=\begin{cases}
p-2,\;&k=1\\
(p-1)^2p^{k-2},\;&k\geq2.
\end{cases}$$
When $k\geq3$,
\begin{align*}
|\mathfrak{A}_{k-1}|&=1+p-2+\sum\limits_{m=2}^{k-1}(p-1)^2p^{m-2}=p-1+(p-1)^2p^{-2}\sum\limits_{m=2}^{k-1}p^m=\\
&=p-1+(p-1)^2p^{-2}\frac{p^2(1-p^{k-2})}{1-p}=\\
&=p-1-(p-1)(1-p^{k-2})=(p-1)p^{k-2}.
\end{align*}
So we have for $k\geq3$
$$\sum\limits_{\theta\in\mathfrak{S}_k}\theta(x)=\begin{cases}
0,&\text{ if }x\in S_0\setminus A_{k-1},\\
-(p-1)p^{k-1},&\text{ if } x\in A_{k-1}\setminus A_k,\\
(p-1)^2p^{k-2},&\text{ if }x\in A_k,
\end{cases}$$
For $k=2$,
$$\sum\limits_{\theta\in\mathfrak{S}_2}\theta(x)=\begin{cases}
0,&\text{ if }x\in S_0\setminus A_{1},\\
-(p-1),&\text{ if } x\in A_1\setminus A_2,\\
(p-1)^2,&\text{ if }x\in A_2,
\end{cases}$$
and for $k=1$,
$$\sum\limits_{\theta\in\mathfrak{S}_1}\theta(x)=\begin{cases}
-1,&\text{ if }x\in S_0\setminus A_{1},\\
p-2,&\text{ if }x\in A_1.
\end{cases}$$

Denote $V_k=A_k\setminus A_{k+1}$. Then for $p^nx$ the following cases are possible: either there exists such $k_x\geq0$ that $p^nx\in V_{k_x}$ (and then this $k_x$ is unique), or $p^nx\in A_k$ for all $k\geq0$, which means $|1-p^nx|_p\leq p^{-k}$ for all $k\geq0$, so $p^nx=1$.

Let $p^nx\in V_0$. We have
$$\sum\limits_{\theta\in\widehat{S_0}}\overline{\theta(p^nx)}e^{-tp^{\alpha\rank\theta}}=-e^{-\lambda_1t}.$$
If $p^nx\in V_1$, then
$$\sum\limits_{\theta\in\widehat{S_0}}\overline{\theta(p^nx)}e^{-tp^{\alpha\rank\theta}}=(p-2)e^{-\lambda_1t}-(p-1)e^{-\lambda_2t}.$$
If $p^nx\in V_2$, then
$$\sum\limits_{\theta\in\widehat{S_0}}\overline{\theta(p^nx)}e^{-tp^{\alpha\rank\theta}}=(p-2)e^{-\lambda_1t}+(p-1)^2e^{-\lambda_2t}-(p-1)p^2e^{-\lambda_3t}.$$
If $p^nx\in V_{k_x}$, $k_0\geq4$, then
\begin{align*}
\sum\limits_{\theta\in\widehat{S_0}}\overline{\theta(p^nx)}e^{-tp^{\alpha\rank\theta}}&=(p-2)e^{-\lambda_1t}+(p-1)^2e^{-\lambda_2t}+\\
&+\sum\limits_{k=3}^{k_x}(p-1)^2p^{k-2}e^{-\lambda_kt}-(p-1)p^{k_x-1}e^{-\lambda_{k_x+1}t}.
\end{align*}
If $p^nx=1$, then
\begin{align*}
\sum\limits_{\theta\in\widehat{S_0}}\overline{\theta(p^nx)}e^{-tp^{\alpha\rank\theta}}&=(p-2)e^{-\lambda_1t}+\sum\limits_{k=2}^{+\infty}(p-1)^2p^{k-2}e^{-\lambda_kt}=\\
&=(p-2)e^{-\lambda_1t}+(p-1)^2p^{-2}\sum\limits_{k=2}^{+\infty}p^{k}e^{-\lambda_kt}.
\end{align*}

Consider
$$\int\limits_{0}^{+\infty}e^{-p^{\alpha x}t}p^xdx.$$
After substitution $u=tp^{\alpha x}$, $du=\alpha t p^{\alpha x}\ln p dx$ we get
$$\int\limits_{0}^{+\infty}e^{-p^{\alpha x}t}p^xdx=\frac{1}{\alpha\ln p}\int\limits_t^{+\infty}e^{-u}\left(\frac{u}{t}\right)^{\frac{1}{\alpha}}\frac{1}{u}du=\frac{1}{\alpha\ln p t^{\frac{1}{\alpha}}}\int\limits_t^{+\infty}e^{-u}u^{\frac{1}{\alpha}-1}du.$$
Hence,
$$|I_1(t,x)|\leq C\left(\sum\limits_{k=1}^{k_x+1}e^{-\lambda_kt}+\frac{1}{\alpha\ln p t^{\frac{1}{\alpha}}}\right).$$
Or, since
$$\int\limits_t^{+\infty}e^{-u}u^{\frac{1}{\alpha}-1}du\leq\int\limits_0^{+\infty}e^{-u}u^{\frac{1}{\alpha}-1}du=\Gamma\left(\frac{1}{\alpha}\right),$$
where $\Gamma$ denotes the gamma function, we get a uniform estimate for all $x\in\mathbb{Q}_p^{\times}$ and $t>0$:
$$|I_1(t,x)|\leq C t^{-\frac{1}{\alpha}}.$$

For the term $I_2$ we have
$$I_2(t,x)=\int\limits_{-\pi/\ln p}^{\pi/\ln p}e^{-K_0(\beta)t}|x|_p^{-i\beta}d\beta=\int\limits_{-\pi/\ln p}^{\pi/\ln p}e^{-\real(K_0(\beta))}e^{-i\im(K_0(\beta))}|x|_p^{-i\beta}d\beta.$$

It follows from \eqref{Re-1} that
$$\real[K_0(\beta)]\geq0,\mbox{ for all }\beta\in\left(-\frac{\pi}{\ln p},\frac{\pi}{\ln p}\right).$$

Since $\abs*{|x|_p^{-i\beta}}=1$, we get the estimate
$$|I_2(t,x)|\leq\int\limits_{-\pi/\ln p}^{\pi/\ln p}|x|_p^{-i\beta}d\beta=\frac{2\pi}{\ln p}.$$
Finally, we have for all $t>0$,
\begin{equation}\Label{gtidt}
|Z(t,x)|\leq C(t^{-1/\alpha}+1).
\end{equation}

Let us consider the dependence on $x$. Suppose that $|x|_p\neq1$, so that $n\neq0$. By \eqref{idt},
$$Z(t,x)=I_2(t,x).$$ 
Put 
$$f_t(\beta)=e^{-tK_0(\beta)}.$$
Since for $K_0(\be)$we have $K_0\left(\frac{\pi}{\ln p}\right)=K_0\left(-\frac{\pi}{\ln p}\right)$, integration by parts gives $$I_2(t,x)=\frac{1}{a}\int\limits_{-\pi/\ln p}^{\pi/\ln p}f_t(\beta)e^{-in\beta\ln p}d\beta=\frac{1}{ian\ln p}\int\limits_{-\pi/\ln p}^{\pi/\ln p}f_t'(\beta)e^{-in\beta\ln p}d\beta.$$
Due to
$$f_t'(\beta)=-tK_0'(\beta)e^{-tK_0(\beta)},$$
and
$$\abs*{e^{-tK_0(\beta)}}\leq1,$$
we obtain
\begin{equation}\Label{gdhm}\ 
|I_2(t,x)|\leq\frac{t}{a|n|\ln p}\int\limits_{-\pi/\ln p}^{\pi/\ln p}|K_0'(\beta)|d\beta\leq\frac{Ct}{|n|}\leq \frac{Ct}{\abs*{\ln|x|_p}},\;|x|_p\neq1,
\end{equation}
since $|n|=\frac{\abs*{\ln|x|_p}}{\ln p}.$

Finally, for $|x|_p\neq1$, we have $I_1(t,x)=0$ by \eqref{idt}. Hence
\[
Z(t,x)=I_2(t,x),
\qquad |x|_p\neq1.
\]
On the other hand, the representation of $I_2$ immediately gives
\[
|I_2(t,x)|
\leq
\frac{1}{a}
\int_{-\pi/\ln p}^{\pi/\ln p}
e^{-t\operatorname{Re}K_0(\beta)}\,d\beta
\leq C,
\qquad |x|_p\neq1,
\]
where we have used $\operatorname{Re}K_0(\beta)\geq0$.
Together with \eqref{gdhm}, this yields
\[
|Z(t,x)|
\leq
C\min\left\{
1,\,
\frac{t}{|\ln |x|_p|}
\right\},
\qquad |x|_p\neq1.
\]

\end{proof}

\begin{proposition}\Label{pr6.2}
The fundamental solution satisfies
$Z(t,\cdot)\in L^2(\mathbb{Q}_p^{\times})$ for all fixed $t>0$.
\end{proposition}

\begin{proof}
It follows from Theorem \ref{th6.1} that for all $t>0$,
\begin{align*}
\int\limits_{\mathbb{Q}_p^{\times}}|Z(t,x)|^2\frac{dx}{|x|_p}&\leq C^2\bigg\{\sum\limits_{n=-\infty}^{-1}\int\limits_{S_n}\frac{t^2}{\abs*{\ln|x|_p}^2}\frac{dx}{|x|_p}+\\
&+\int\limits_{S_0}\abs*{t^{-1/\alpha}+1}^2\frac{dx}{|x|_p}
+\sum\limits_{n=1}^{+\infty}\int\limits_{S_n}\frac{t^2}{\abs*{\ln|x|_p}^2}\frac{dx}{|x|_p}\bigg\}.
\end{align*}
For $n<0$ we have
$$\int\limits_{S_n}\frac{1}{\ln^2|x|_p}\frac{dx}{|x|_p}=\frac{1}{\ln^2p}\int\limits_{S_n}\frac{1}{n^2}\frac{dx}{p^n}=\frac{1-p^{-1}}{n^2\ln^2p},$$
so that
$$\sum\limits_{n=-\infty}^{-1}\int\limits_{S_n}\frac{t^2}{\abs*{\ln|x|_p}^2}\frac{dx}{|x|_p}\leq\frac{1-p^{-1}}{n^2\ln^2p}\sum\limits_{n=1}^{\infty}\frac{1}{n^2}<\infty.$$
Similarly,
$$\sum\limits_{n=1}^{+\infty}\int\limits_{S_n}\frac{t^2}{\abs*{\ln|x|_p}^2}\frac{dx}{|x|_p}\leq\frac{1-p^{-1}}{n^2\ln^2p}\sum\limits_{n=1}^{\infty}\frac{1}{n^2}<\infty.$$
We have
$$\int\limits_{S_0}\frac{dx}{|x|_p}=1-p^{-1}.$$
This concludes the proof.
\end{proof}

\begin{proposition}\Label{pr6.3}
The function $u$ defined in \eqref{6-sol} is the solution of the problem \eqref{6.1}--\eqref{6.2}.
\end{proposition}
\begin{proof}
Follows from the properties of $(C_0)$--semigroups.
\end{proof}

\section{Convolution semigroup associated with $W^\alpha$}

In this section we construct the convolution semigroup associated with the operator $W^\alpha$ and show its connections with the semigroup $T_t$ constructed in Section~\ref{Sec5}.

First we remind some standard definition (see e.g. \cite{B-F}):
\begin{definition}\Label{Def9-1}\
Let $G$ be a locally compact Abelian group and let
$\psi:G\to\mathbb C$. The function $\psi$ is called {\it negative definite} if
\[\sum_{j,k=1}^{n} \left(\psi(x_j)+\overline{\psi(x_k)}
-\psi(x_j-x_k)\right)c_j\overline{c_k}\ge 0\]
for every choice of points $x_1,\ldots,x_n\in G$ and complex numbers $c_1,\ldots,c_n\in\mathbb C.$
\end{definition}
Due to  \cite[Proposition 7.5]{B-F} this definition is equivalent to the following:

\begin{definition}
A function $\psi:G\to\mathbb C$ with $\psi(e)=0$
is negative definite if and only if
\[\sum_{j,k=1}^{n} \psi(x_j^{-1}x_k) c_j\overline{c_k}
\le 0 \]
for all $x_1,\ldots,x_n\in G$ and $c_1,\ldots,c_n\in\mathbb C$ satisfying $ \sum\limits_{j=1}^{n} c_j=0. $
\end{definition}

\begin{theorem}[Negative definiteness of the symbol]\Label{th-neg}
Let $\pi\in\widehat{\mathbb Q_p^\times},$
where $K(\pi)$ is given by \eqref{3.3}.
\[K(\pi)
=
c_\alpha
\int_{\mathbb Q_p^\times}
\frac{\pi(y)-1}{|1-y|_p^{\alpha+1}} \,d y,
\qquad 0<\alpha<1.\]
Then $K$ is a negative definite function on
$\widehat{\mathbb Q_p^\times}$.
\end{theorem}
\begin{proof} 1) Let $e$ denote the trivial character of $\mathbb Q_p^\times$. Then $K(e)=0.$ Indeed, by \eqref{3.3}we have,
\[
K(e) = c_\alpha \int_{\mathbb Q_p^\times}
\frac{e(y)-1}{|1-y|_p^{\alpha+1}} \,d y. \]
Since $e(y)\equiv 1,$  we have $e(y)-1=0$ for all $y\in\mathbb Q_p^\times$. Therefore, $K(e) =0$.

2) Let $ \pi_1,\ldots,\pi_n\in\widehat{\mathbb Q_p^\times}
$ and let $c_1,\ldots,c_n\in\mathbb C $ satisfy
\[\sum_{j=1}^n c_j=0.\]
Consider
\[S=\sum_{j,k=1}^{n} K(\pi_j\pi_k^{-1}) \,c_j\overline{c_k}. \]
Using the definition of $K$, we obtain
\[
S = c_\alpha \sum_{j,k=1}^{n} c_j\overline{c_k}
\int_{\mathbb Q_p^\times} \frac{\pi_j(y)\overline{\pi_k(y)} -1}{|1-y|_p^{\alpha+1}} \,d y .
\]
We may interchange summation
and integration:
\[S = c_\alpha \int_{\mathbb Q_p^\times}
\frac{\sum_{j,k=1}^{n}c_j\overline{c_k}
\bigl(\pi_j(y)\overline{\pi_k(y)}-1\bigr)
}{|1-y|_p^{\alpha+1}} \,d y .
\]
Now
\[\sum_{j,k=1}^{n} c_j\overline{c_k} = \left|
\sum_{j=1}^{n}c_j \right|^2 = 0,\]
hence
\[\sum_{j,k=1}^{n} c_j\overline{c_k} \bigl(
\pi_j(y)\overline{\pi_k(y)}-1 \bigr) = \sum_{j,k=1}^{n}
c_j\overline{c_k} \pi_j(y)\overline{\pi_k(y)}.
\]
Since
\[\sum_{j,k=1}^{n} c_j\overline{c_k} \pi_j(y)\overline{\pi_k(y)} = \left| \sum_{j=1}^{n} c_j\pi_j(y)
\right|^2, \]
we arrive at
\[S = c_\alpha \int_{\mathbb Q_p^\times} \frac{
\left| \sum_{j=1}^{n} c_j\pi_j(y) \right|^2 }
{|1-y|_p^{\alpha+1}} \,d y .\]
Since $c_\alpha <0, $ and the integrand is nonnegative, it follows that $S\le 0.$ Therefore $K$ is negative definite.
\end{proof}

\begin{definition}
A family $(\mu_t)_{t>0}$ of positive bounded measures on $\mathbb Q_p^n$ with the
properties

\begin{enumerate}
\item[(i)] $\mu_t(G)\le 1, \quad \text{for } t>0,$
\item[(ii)] $ \mu_t * \mu_s = \mu_{t+s}, \quad \text{for } t,s>0,$
\item[(iii)] $ \lim_{t\to 0}\mu_t=\delta_e$ vaguely,
\end{enumerate}
is called a {\it vaguely continuous convolution semigroup} on $\mathbb Q_p^n$. Here $\delta_e$ denotes the Dirac measure at the identity element
$e\in \mathbb Q_p^\times$.
\end{definition}

Let us remind that the family of Radon measures $(\mu_\lambda)_{\lambda\in\Lambda}$
on the locally compact Hausdorff space  $X$ converges {\it vaguely} to a Radon measure $\mu$, if
\[\lim_{\lambda} \int_X f(x)\,d\mu_\lambda(x) =
\int_X f(x)\,d\mu(x)\]
for every function $f\in C_c(X),$
where $C_c(X)$ denotes the space of continuous functions on $ X$
with compact support.

\medskip

\begin{theorem}
Let $K$ denote the symbol of the operator $W^\alpha$ given by \eqref{3.3}.
There exists a unique convolution semigroup $(\mu_t)_{t>0}$ of probability measures on \(\mathbb Q_p^\times\) such that
\[ \widehat{\mu_t}(\pi) = e^{-tK(\pi)},
\qquad t>0, \quad \pi\in\widehat{\mathbb Q_p^\times}.
\]
Equivalently,
\[\widehat{\mu_t}(\pi) = \exp\!\left(
-t\,c_\alpha \int_{\mathbb Q_p^\times}
\frac{\pi(y)-1} {|1-y|_p^{\alpha+1}}
\,dy \right). \]
\end{theorem}
\begin{proof} Remark that from \cite[Theorem 8.3]{B-F} it follows that any continuous negative defined function on the duality group defines convolution semigroup. Given that due to the Theorem~\ref{th-neg} the function $K$ is negative defined, it remains to prove its continuity.

Let us first note that for every $k\ge1$, the annihilator to  $A_k=1+p^k\mathbb Z_p$:
\[A_k^\perp = \{\pi\in\widehat{\mathbb Q_p^\times}:\pi(x)=1,\ \forall x\in A_k\} \]
is open in $\widehat{\mathbb Q_p^\times}$. This follows from Statement~2.9 in \cite{B-F} and that $A_k$ is a compact subgroup of the locally compact Abelian group $Q_p^\times$.

This immediately implies that the rank function $R:\widehat{\mathbb Q_p^\times}\longrightarrow\mathbb N_0,$  $R(\pi)=\operatorname{rank}\pi, $ is continuous.
Ideed, by definition,
\[R^{-1}(\{k\}) = A_k^\perp\setminus A_{k-1}^\perp,
\qquad k\ge1,\]
while $R^{-1}(\{0\}) = A_0^\perp.$
Since each annihilator $A_k^\perp$ is open, every set $R^{-1}(\{k\})$ is open too. Therefore $R$ is continuous, because $\mathbb N_0$ is endowed with the discrete topology.

This finally gives that the function
\[
K:\widehat{\mathbb Q_p^\times}\longrightarrow\mathbb R,
\qquad
K(\pi)=p^{\alpha\,\operatorname{rank}(\pi)},
\]
is continuous as superposition of continuous maps, which together with continuity of  $K_0(\beta)$ w.r.t. $\beta$ finishes the proof.
\end{proof}

\begin{corollary}
For $\varphi\in L^2(\mathbb{Q}_p^{\times})$, let
$$(T_t\varphi)(x)=\int\limits_{\mathbb{Q}_p^{\times}}\varphi(xy^{-1})\mu_t(dy),\quad x\in\mathbb{Q}_p^{\times},t>0.$$
Then $(T_t)_{t>0}$ is a contraction $(C_0)$--semigroup and
$$(\mathfrak{M}T_t\varphi))(\pi)=e^{-tK(\pi)}(\mathfrak{M}\varphi)(\pi),\;\pi\in\widehat{\mathbb{Q}_p^{\times}}.$$
Moreover, its infinitesimal generator is given by $A=-W^{\alpha}$.
\end{corollary}

\begin{proof}
By the Berg--Forst theorem, there exists a unique convolution semigroup of probability
measures $(\mu_t)_{t\ge0}$ on $Q_p^\times$ satisfying
\[
\widehat{\mu_t}(\pi)=e^{-tK(\pi)}, \qquad \pi\in\widehat{Q_p^\times}.
\]
The theorem associates with $(\mu_t)_{t\ge0}$ a contraction convolution semigroup
$(S_t)_{t\ge0}$ whose Mellin representation is given by
\[
\mathfrak M (S_t f)(\pi)=e^{-tK(\pi)}\,\mathfrak M f(\pi),
\qquad f\in L^2(Q_p^\times).
\]

On the other hand, by Theorem~\ref{SemTt}, the operator $(-W^\alpha)$ generates a contraction
$C_0$-semigroup $(T_t)_{t\ge0}$ satisfying
\[
\mathfrak M(T_t f)(\pi)=e^{-tK(\pi)}\,\mathfrak M f(\pi).
\]
Thus $\mathfrak M(S_t f)=\mathfrak M(T_t f).$
Since the Mellin transform is injective, we conclude that
$S_t=T_t.$ Therefore,
\[
(T_t f)(x)
=
\int_{Q_p^\times}f(xy^{-1})\,\mu_t(dy),
\qquad x\in Q_p^\times,\; t\ge0.
\]
\end{proof}

\bigskip
\section*{Acknowledgements}
 The first author acknowledges the funding support in the framework of the project ``Spectral Optimization: From Mathematics to Physics and Advanced Technology'' (SOMPATY) received from the European Union’s Horizon 2020 research and innovation programme under the Marie Skłodowska-Curie grant agreement No 873071. The second and third authors acknowledge the financial support by the National Research Foundation of Ukraine (Project number 2023.03/0002).
 
The authors are grateful to Professor Roman Urban for pointing out an error in an earlier version of this manuscript, which allowed us to fix it and obtain a better estimate.

\end{document}